\documentclass[10pt]{IEEEtran}
\usepackage{amsmath, mathrsfs,amsfonts}
\usepackage{amssymb}
\usepackage{acronym}
\usepackage{graphicx, amsthm}
\usepackage{pdfsync}
\usepackage{color}

\newcommand{\BEQA}{\begin{eqnarray}}
\newcommand{\EEQA}{\end{eqnarray}}

\def\e{\epsilon}
\def\a{\alpha}
\def\p{\hat p}

\def \foral {\textrm{for all }}
\def \pr  {\mathbf{P}}
\def \R  {\mathbb{R}}
\def \E  {\mathbf{E}}

\newtheorem{theorem}{Theorem}

\newtheorem{corollary}{Corollary}
\newtheorem{lemma}{Lemma}

\theoremstyle{definition}

\newtheorem{definition}{Definition}

               {\begin{list}{}{\leftmargin#1\rightmargin#2\topsep#3}\item{}}%
               {\end{list}}

\title{Access-Network Association Policies for Media Streaming in Heterogeneous Environments}

\author{Ali ParandehGheibi$^\dag$\thanks{$\dag$ Dept. of Electrical Engineering and Computer Science, MIT.}, Muriel M\'edard$^\dag$, Asuman Ozdaglar$^\dag$, Srinivas Shakkottai$^\ddag$ \thanks{$\ddag$ Dept. of Electrical and Computer Engineering, Texas A\&M University.}\\
\normalsize{parandeh@mit.edu, medard@mit.edu, asuman@mit.edu, sshakkot@tamu.edu}}

\begin{document}

\maketitle

\begin{abstract}

We study the design of media streaming applications in the presence of multiple heterogeneous wireless access methods with different throughputs and costs.  Our objective is to analytically characterize the trade-off between the usage cost and the Quality of user Experience (QoE), which is represented by the probability of interruption in media playback and the initial waiting time.
We model each access network as a server that provides packets to the user according to a Poisson process with a certain rate and cost.  Blocks are coded using random linear codes to alleviate the duplicate packet reception problem. Users must take decisions on how many packets to buffer before playout, and which networks to access during playout.

We design, analyze and compare several control policies with a threshold structure. We formulate the problem of finding the optimal control policy as an MDP with a probabilistic constraint. We present the HJB equation for this problem by expanding the state space, and exploit it as a verification method for optimality of the proposed control law.

\end{abstract}

\section{Introduction}\label{introduction_sec}

Media streaming is fast becoming the dominant application on the Internet \cite{Lab09}.  The popularity of such media transfers has been accompanied by the growing usage of wireless handheld devices as the preferred means of media access.  It is expected that such media streaming would happen in both a device to device (D2D) as well as in a base-station to device fashion, and both the hardware and applications needed for such communication schemes are already making an appearance \cite{Lar10,knocking}.

Media streaming is achieved by dividing a file into \emph{blocks,} which are then further divided into packets for transmission.  After each complete block is received, it can be decoded and played out.  Since we consider a streaming application, blocks inherently have a sequence associated with them, and each block must be received by the time the previous one has been played out.  The absence of a block at the time of playout would cause a frame freeze, which is to be avoided if possible.  When there are multiple networks that can be used to access a particular piece of content (e.g. from a base station or a peer device) each device must take decisions on associating with one or more such access networks.  However, the costs of different access methods might be different.  For example, accessing the base station of a cellular network can result in additional charges per packet, while it might be possible to receive the same packets from the access point of a local WLAN or another device with a lower cost or possibly for free.  Further, the cost of communication might be mitigated by the initial amount of buffering before playout.  Hence, there are  trade-offs between the probability of frame skipping, the initial waiting time, and the cost of different access methods available.

The objective of this paper is to understand the trade-off between initial buffering, and the usage of low and costly communication methods for attaining a target probability of skip-free playout.  We consider a system wherein network coding is used to ensure that packet identities can be ignored, and packets may potentially be obtained from two sources (servers) that have different rates of transmission.  The wireless channel is unreliable, and we assume that each server can deliver packets according to a Poisson process with a known rate. Further, the costs of accessing the two servers are different; for simplicity we assume that one of the servers is free.  Thus, \emph{our goal is to develop an algorithm that switches between the free and the costly servers in order to attain a target probability of skipping at lowest cost.}

Our contributions are as follows. We first develop an analytical characterization of the interruption probability for the single server case.  Using this result, we obtain a lower bound on the cost of offline policies that do not observe the trajectory of packets received.  We show that such policies have a threshold form in terms of the time of association with the costly server.  Using the offline algorithm as a starting point, we develop an \emph{online} algorithm with lower cost that has a threshold form -- both free and costly servers are used until the queue length reaches a threshold, followed by only free server usage.  We then develop an online algorithm in which the risk of interruption is spread out across the trajectory.  Here, only the free server is used whenever the queue length is above a certain threshold, while both servers are used when the queue length is below the threshold. The threshold is designed as a function of the initial buffer size and the desired interruption probability.

We formulate the problem of finding the optimal network association policy  as a Markov Decision Process with a probabilistic constraint. Similarly to the Bellman equation proposed by Chen \cite{Chen04} for a discrete time MDP with probabilistic constraints, we write the Hamilton-Jacobi-Bellman equation for the problem. Using a guess and check approach, we derive an approximate solution of the HJB equation, and show that the optimal policy given by the approximate value function takes a threshold form.

Media streaming, particularly in the area of P2P networks has attracted significant recent interest. For example, work such as  \cite{ZhoChi07,BonMas08,YinSri10} develop analytical models on the trade-off between the steady state probability of missing a block and buffer size under different block selection policies.  Unlike our model, they consider live streaming with deterministic channels.  The use of random linear codes considerably simplifies packet selection \cite{Acedanski05, Rodriguez06, wangLi07, ChiZhang06}, and we can use the same idea to ensure that packets can be received from multiple sources without the need to coordinate the exact identities of the packets from each.  However, we focus on content that is already cached at multiple locations, and must be streamed over one or more unreliable channels.  Related to our work is \cite{KumAlt07}, which considers two possible wireless access methods (WiFi and UMTS) for file delivery, assuming particular throughput models for each access method.  In contrast to this work, packet arrivals are stochastic in our model, and our streaming application requires hard constraints on quality of user experience.

\section{System Model and QoE Metrics}
We consider a media streaming system as follows. A single user is receiving a media file of size $F$ from various servers it is connected to. Each server could be a wireless access point or another wireless user operating as a server. The receiver first buffers $D$ packets from the beginning of the file, and then starts the playback.

We assume that time is continuous, and the arrival process of packets from each server is a Poisson process independent of other arrival processes. Further, we assume that no redundant packet is delivered from different servers. This assumption can be justified if there is no delay in the feedback to the servers, or by sending random linear combination of the packets in the server (see \cite{CheLi08} and \cite{ISIT_report} for more details). Therefore, we can combine the arrival processes of any subset $S$ of the servers into one Poisson process of rate $R_S$ equal to the summation of the rates from the corresponding servers.

 There are two types of servers in the system: free servers and the costly ones. There is no cost associated with receiving packets from a free server, but a unit cost is incurred for each (coded) packet delivered by any costly server. As described above, we can combine all the free servers into one free server from which packets arrive according to a Poisson process of rate $R_0$. Similarly, we can merge all of the costly servers into one costly server with effective rate of $R_c$. At any time $t$, the user has the option to use packets only from the free server or from both the free and the costly servers. In the latter case, the packets arrive according to a Poisson process of rate $R_1 = R_0 + R_c$.  The user's action at time $t$ is denoted by $u_t \in \{ 0,1 \}$, where $u_t =0$ if only the free server is used at time $t,$ while $u_t =1$ if both free and costly servers are used. We normalize the playback rate to one, i.e., it takes one unit of time to play a single packet.  We also assume that the parameters $R_0$ and $R_1$ are known at the receiver.

 The dynamics of the receiver's buffer size (queue-length) $x_t$ can be described as follows
\begin{equation}\label{buffer}
    x_t =  D + N_t + \int_0^t u_\tau dN^c_\tau  - t,
\end{equation}
where $D$ is the initial buffer size, $N_t$ Poisson processes of rate $R_0$ and $N^c_t$ is a Poisson counter of rate $R_c$ which is independent of the process $N_t$. The last term correspond to the unit rate of media playback.

The user's association (control) policy is formally defined below.
\begin{definition}\label{policy_def}
[Control Policy] Let $h_t = \{x_s: 0\leq s\ \leq t\} \cup \{u_s: 0\leq s\ < t\}$ denote the history of the buffer sizes and actions up to time $t$, and $\mathcal H$ be the set of all histories for all $t$. A \emph{deterministic association policy } denoted by $\pi$ is a mapping $\pi : \mathcal H \longmapsto \{ 0,1 \} $, where at any time $t$
$$\pi(h_t) = \left\{
             \begin{array}{ll}
               0, & \hbox{if only the free server is chosen,} \\
               1, & \hbox{if both servers are chosen.}
             \end{array}
           \right.
$$
Denote by $\Pi$ the set of all such control policies.
\end{definition}

We can declare an interruption in playback when the buffer size decreases to zero before reaching the end of the file, i.e., when there is no packet at the receiver to be played but the file is not completely downloaded. More precisely, let
\begin{eqnarray}\label{hitting_def}
\tau_e =\inf \{t: x_t \leq 0 \}, \quad
\tau_f = \inf \{t: x_t \geq F - t \},&&
\end{eqnarray}
where $\tau_f$ corresponds to time of completing the file download, because we have already played $\tau_f$ packets and the buffer contains the remaining $F-\tau_f$ packets to be played.  The video streaming is interrupted if and only if $\tau_e < \tau_f$.

We consider the following metrics to quantify Quality of user Experience (QoE). The first metric is the initial waiting time before the playback starts. This is directly captured by the initial buffer size $D$. Another metric that affects QoE is the probability of interruption during the playback for a particular control policy $\pi$ denoted by
\begin{equation}\label{p_int}
p^{\pi}(D) = \pr \{\tau_e < \tau_f\},
\end{equation}
 where $\tau_e$ and $\tau_f$ are defined in (\ref{hitting_def}).

\begin{definition}\label{feas_policy_def}
The policy $\pi$ is defined to be  $(D,\epsilon)$-\emph{feasible} if $p^{\pi}(D) \leq \epsilon$. The set of all such feasible policies is denoted by $\Pi(D,\epsilon)$.
\end{definition}

The third metric that we consider in this work is the expected cost of using the costly server which is proportional to the expected usage time of the costly server. For any $(D,\epsilon)$, the usage cost of a $(D,\epsilon)$-feasible policy $\pi$ is given by\footnote{Throughout this work, we use the convention that the cost of an infeasible policy is infinite.}
\begin{equation}\label{cost_def}
J^{\pi}(D,\e) = \E \Big[\int_0^F u_t dt\Big].
\end{equation}

The value function or optimal cost function $V$ is defined as
\begin{equation}\label{value_def}
    V(D, \e) = \min_{\pi \in \Pi(D,\epsilon)} J^{\pi}(D,\e),
\end{equation}
and the optimal policy $\pi^*$ is defined as the optimal solution of the minimization problem in (\ref{value_def}).

In our model, the user expects to have an interruption-free experience with probability higher than a desired level $1-\epsilon$. Note that there is a  fundamental trade-off  between the interruption probability $\epsilon$, the initial buffer size $D$, and the usage  cost. These trade-offs depend on the association policy as well as the system parameters $R_0$, $R_c$ and $F$.

We first characterize the trade-offs between the QoE metrics for degenerate control policies. Next, we use these results to design association policies.

\section{QoE Trade-offs for the Single-Server Problem}
Consider a single-server problem where the receiver receives the packets according to a Poisson process of rate $R$. The user's only decision in this case is the initial buffer size $D$. We would like to  characterize the optimal trade-off between the initial buffer size and the interruption probability $p(D)$ by providing bounds on the interruption probability as a function of the system parameters $R$ and $F$. An upper bound (achievability) on $p(D)$ is particularly useful, since it provides a sufficient condition for desirable user experience. A lower bound (converse) on $p(D)$ demonstrates how tight the upper bound is.

\begin{theorem}\label{pd_thm}
For the initial buffer size $D$, let $p(D)$ be the interruption probability of a single-server system defined as in (\ref{p_int}). Define $\gamma(r)$  as
\begin{equation}\label{gamma_def}
 \gamma(r) = r + R(e^{-r} -1 ),
\end{equation}
and $\bar r(R)$ as the largest root of $\gamma(r)$, i.e.,
\begin{equation}\label{r_bar}
    \bar r(R) = \sup\{r: \gamma(r) = 0 \}.
\end{equation}

Then for all $R>1$,
  \begin{equation}\label{De_to}
   e^{-\bar r(R) D} - 2e^{-\frac{(R-1)^2}{4(R+1)}F}  \leq p(D) \leq      e^{-\bar r(R) D} .
  \end{equation}
\end{theorem}
\begin{proof}
We do not include the proof owing to space limitations. See \cite{ISIT_report} for a complete proof.
\end{proof}

Note that the upper bounds and lower bounds of $p(D)$ given by Theorem \ref{pd_thm} are asymptotically tight as $F$ goes to infinity. Therefore, for $F = \infty$, by continuity of the probability measure we get
\begin{equation}\label{pd_exact}
 p(D) = \pr\Big(\min_{t\geq 0} x_t \leq 0 \big| x_0 = D\Big)= e^{-\bar r(R) D}.
\end{equation}

Using this characterization, we can identify the ranges of the QoE metrics for which there exists no feasible policy or the costly server is not required.
\begin{corollary}\label{boundary_cor}

(a) For any $(D,\epsilon)$ such that $D\geq \frac{1}{\bar r(R_0)}\log\big(\frac1\epsilon\big)$,
$$\min_{\pi \in \Pi} J^{\pi}(D,\e)  = 0.$$

(b) For any $(D,\epsilon)$ such that $D < \frac{1}{\bar r(R_1)}\log\big(\frac1\epsilon\big)$,
$$\min_{\pi \in \Pi} J^{\pi}(D,\e) = \infty.$$

\end{corollary}
\begin{proof}
Consider the degenerate policy $\pi_0 \equiv 0$. This policy is equivalent to a single-server system with arrival rate $R = R_0$. By Definition \ref{feas_policy_def}, and (\ref{pd_exact}), the policy $\pi_0$ is $(D,\epsilon)$-feasible for all $D\geq \frac{1}{\bar r(R_0)}\log\big(\frac1\epsilon\big)$. Note that by (\ref{cost_def}) this policy does not incur any cost, which results in part (a).

Moreover, for all $(D,\epsilon)$ with $D < \frac{1}{\bar r(R_1)}\log\big(\frac1\epsilon\big)$, there is no $(D,\epsilon)$-feasible policy. This is so since the buffer size under any policy $\pi$ is stochastically dominated by the one governed by the degenerate policy  $\pi_1 \equiv 1$. Hence,
$$p^{\pi}(D) \geq p^{\pi_1}(D) = \exp(-\bar r(R_1) D) > \epsilon.$$
Using the convention of infinite cost for infeasible policies, we obtain the result in part (b).
\end{proof}

For simplicity of notation, let $\alpha_0 = \bar r(R_0)$, and $\alpha_1 = \bar r(R_1)$.  Throughout the rest of this paper, we study the case that the file size $F$ is infinite, since the control policies in this case take simpler forms and the cost of such control policies provide an upper bound for the finite file size case. Further, by Corollary \ref{boundary_cor} we focus on the region
\begin{equation}\label{S_def}
\mathcal R = \Big\{(D,\epsilon): \frac{1}{\alpha_1}\log\big(\frac1\epsilon\big) \leq D \leq \frac{1}{\alpha_0}\log\big(\frac1\epsilon\big)\Big\}
\end{equation}
to analyze the expected cost of various classes of control policies.

\section{Design and Analysis of  Association Policies}
In this section, we propose several classes of parameterized control policies. We first characterize the range of the parameters for which the association policy is feasible for a given initial buffer size $D$ and the desired level of interruption probability $\epsilon$. Then, we try to choose the parameters such that the expected cost of the policy is minimized.

\subsection{Off-line Policy}
Consider the class of policies where the decisions are made off-line before starting media streaming. In this case, the arrival process is not observable by the decision maker. Therefore, the user's decision space reduces to the set of deterministic functions $u: \R \rightarrow \{0,1\}$, that maps time into the action space.

\begin{theorem}\label{offline_form_thm}
Let the cost of a control policy be defined as in (\ref{cost_def}). In order to find a minimum-cost off-line policy, it is sufficient to consider policies of the following form:
\begin{equation}\label{offline_form}
    \pi(h_t) = u_t = \left\{
            \begin{array}{ll}
              1, & \hbox{if $t \leq t_s$} \\
              0, & \hbox{if $t > t_s.$}
            \end{array}
          \right.
\end{equation}
\end{theorem}

\begin{proof}
In general any off-line policy $\pi$ consists of multiple intervals in which the costly server is used. Consider an alternative policy $\pi'$  of the form of (\ref{offline_form}) where $t_s = J^{\pi}$. By definition of the cost function in (\ref{cost_def}) the two policies incur the same cost. Moreover, the buffer size process under policy $\pi$ is stochastically dominated by the one under policy $\pi'$, because the policy $\pi'$ counts the arrivals from the costly server earlier, and the arrival process is stationary. Hence, the interruption probability of $\pi'$ is not larger than that of $\pi$. Therefore, for any off-line policy, there exists another off-line policy of the form given by (\ref{offline_form}).
\end{proof}

\begin{theorem}\label{offline_range_thm}
Consider the class of off-lines policies of the form (\ref{offline_form}). For any $(D,\e) \in \mathcal R$, the policy $\pi$ defined in (\ref{offline_form}) is feasible if
\begin{equation}\label{ts_range}
    t_s \geq t_s^* = \frac{R_0}{R_1 - R_0} \bigg[ \frac{1}{\alpha_0} \log\Big( \frac{1}{\e - e^{-\alpha_1 D}}  \Big) - D \bigg].
\end{equation}
\end{theorem}

\begin{proof}
By Definition \ref{feas_policy_def}, we need to show that $p^\pi(D) \leq \e$. By a union bound on the interruption probability, it is sufficient to verify
\begin{equation}\label{union_bd}
\pr\Big(\min_{0\leq t\leq t_s} x_t \leq 0 \big| x_0 = D\Big) + \pr\Big(\min_{t > t_s} x_t \leq 0 \big| x_0 = D\Big) \leq \e.
\end{equation}

In the interval $[0,t_s]$,  $x_t$ behaves as in a single-server system with rate $R_1$. Hence, by Theorem \ref{pd_thm} we get
\begin{equation}\label{pint_bd1}
\pr\Big(\min_{0\leq t\leq t_s} x_t \leq 0 \big| x_0 = D\Big) \leq e^{-\alpha_1 D}.
\end{equation}

For the second term in (\ref{union_bd}), we have

\begin{eqnarray*}
&&  \pr\Big(\min_{t > t_s} x_t \leq 0 \big| x_0 = D\Big)  \\
&&  = \sum_{q = D-t_s}^\infty \pr\Big(\min_{t > t_s} x_t \leq 0 \big| x_{t_s} = q\Big) \pr(x_{t_s} = q) \\
&&  \stackrel{(a)}{\leq} \sum_{q = D-t_s}^\infty e^{-\a_0 q} \pr(x_{t_s} = q) \\
&&  = \sum_{k = 0}^\infty e^{-\a_0 (D+k-{t_s})} \pr(N_{t_s} + N^c_{t_s} = k) \\
&&  \stackrel{(b)}{=} \sum_{k = 0}^\infty e^{-\a_0 (D+k-{t_s})} \frac{e^{-R_1 {t_s}}(R_1 {t_s})^k}{k!} \\
&&  = e^{-\a_0 (D-{t_s})+R_1 {t_s} (e^{-\a_0}-1)} \sum_{k = 0}^\infty \frac{e^{-R_1 {t_s} e^{-\a_0}}(R_1 {t_s} e^{-\a_0})^k}{k!}\\
&& = \exp\Big(-\a_0 (D-{t_s})+R_1 {t_s} (e^{-\a_0}-1)\Big)\cdot 1\\
&& \stackrel{(c)}{=} \exp\Big(-\a_0 (D-{t_s})+R_1 {t_s} (-\frac{\a_0}{R_0})\Big)\\
&& \stackrel{(d)}{\leq} \e - e^{-\alpha_1 D},
\end{eqnarray*}
where (a) follows from Theorem \ref{pd_thm} and the fact that $u_t = 0$, for $t \geq t_s$. (b) is true because $N_{t_s} + N^c_{t_s}$ is a Poisson random variable with mean $R_1 t_s$. (c) holds since $\a_0$ is the root of $\gamma(r)$ defined in (\ref{gamma_def}) for $R = R_0$. Finally, (d) follows from the hypothesis of the theorem.

By combining the above bounds, we may verify (\ref{union_bd}) which in turns proves feasibility of the proposed control policy.
\end{proof}

Note that obtaining the optimal off-line policy is equivalent to finding the smallest $t_s$ for which the policy is still feasible. Therefore, $t_s^*$ given in (\ref{ts_range}) provides an upper bound on the minimum cost of an off-line policy. Observe that $t^*_s$ is almost linear in $D$ for all $(D,\e)$ that is not too close to the lower boundary of region $\mathcal R$. As $(D,\e)$ gets closer to the boundary, $t^*_s$ and the expected cost grows to infinity, which is in agreement with Corollary \ref{boundary_cor}. In this work we pick $t^*_s$ as a benchmark for comparison to other policies that we  present next.

\subsection{Online Safe Policy}
Let us now consider the class of online policies where the decision maker can observe the buffer size history. Inspired by the structure of the optimal off-line policies, we first focus on a \emph{safe} control policy in which in order to avoid interruptions, the costly server is used  at the beginning  until the buffer size reaches a certain threshold after which the costly server is never used. This policy is formally defined below.

\begin{definition}\label{safe_def}
The online safe policy $\pi^S$ parameterized by the threshold value $S$ is given by
\begin{equation}\label{safe_policy}
    \pi^S(h_t) = \left\{
                   \begin{array}{ll}
                     1, & \hbox{if $t \leq \tau_s$} \\
                     0, & \hbox{if $t > \tau_s$,}
                   \end{array}
                 \right.
\end{equation}
where $\tau_S = \inf \{t \geq 0: x_t \geq S\}$.
\end{definition}

\begin{theorem}\label{safe_thm}
Let $\pi^S$ be the safe policy defined in Definition \ref{safe_def}.
For any $(D,\e) \in \mathcal R$, the safe policy is feasible if
\begin{equation}\label{S_range}
    S \geq S^* = \frac{1}{\alpha_0} \log\Big( \frac{1}{\e - e^{-\alpha_1 D}}  \Big).
\end{equation}
Moreover,
\begin{eqnarray*}\label{safe_cost}
    \min_{S\geq S^*} J^{\pi^S}(D,\e)\!\!\!\! &=& \!\!\!\! J^{\pi^{S^*}}(D,\e) \\
   &=& \!\!\!\! \frac{1}{R_1 - 1} \bigg[ \frac{1}{\alpha_0} \log\Big( \frac{1}{\e - e^{-\alpha_1 D}}  \Big) - D +\xi\bigg],
\end{eqnarray*}
where $\xi \in [0,1)$.
\end{theorem}
\begin{proof}
Similar to the proof of Theorem \ref{offline_range_thm}, we need to show that the total probability of interruption before and after crossing the threshold $S$ is bounded from above by $\e$. Observe that for any realization of $\tau_S$ the bound in (\ref{pint_bd1}) still holds. Further, since the costly server is not used after crossing the threshold and $x_{\tau_S} \geq S$, Theorem \ref{pd_thm} implies
\begin{eqnarray}\label{pint_bd2}
\pr\Big(\min_{t > \tau_S} x_t \leq 0 \big| x_0 = D\Big)   \leq  e^{-\alpha_0 S} \leq \e -  e^{-\alpha_1 D},&&
\end{eqnarray}
where the second inequality follows from (\ref{S_range}). Finally, combining (\ref{pint_bd1}) and (\ref{pint_bd2}) gives $p^{\pi^S}(D) \leq \e$, which is the desired feasibility result.

For the second part, first observe that $J^{\pi^{S}}(D,\e) = \E[\tau_S]$. In order to cross a threshold $S \geq S^*$, the threshold $S^*$ must be crossed earlier, because $x_0 = D \leq S^*$. Hence, $\tau_S$ stochastically dominates $\tau_S^*$, implying
$$J^{\pi^{S}}(D,\e) = \E[\tau_S] \geq \E[\tau_{S^*}] =J^{\pi^{S^*}}(D,\e), \quad \foral S\geq S^*.$$

It only remains to compute $\E[\tau_{S^*}]$. It follows from Wald's identity or Doob's optional stopping theorem \cite{SP_book} that
\begin{equation}\label{Walds_stopping}
D + (R_1 - 1)\E[\tau_{S^*}] = \E[x_{\tau_{S^*}}] = S^* + \xi,
\end{equation}

where $\xi \in [0,1)$ because the jumps of a Poisson process are of units size, and hence the overshoot size when crossing a threshold is bounded by one, i.e.,  $S^* \leq x_{\tau_{S^*}} < S^*+1$. Rearranging the terms in (\ref{Walds_stopping}) and plugging the value of $S^*$ from (\ref{S_range}) immediately gives the result.
\end{proof}

Let us now compare the online safe policy $\pi^{S^*}$ with the off-line policy defined in (\ref{offline_form}) with parameter $t_s^*$ as in (\ref{ts_range}). We observe that the cost of the online safe policy is almost proportional to that of the off-line policy, where the cost ratio of the off-line policy to that of the online safe policy is given by
$$\frac{R_0(R_1-1)}{R_1-R_0} = 1 + \frac{R_1(R_0-1)}{R_1-R_0} > 1.$$
Note that the structure of both policies is the same, i.e, both policies use the costly server for a certain period of time and then switch back to the free server. As suggested here, the advantage of observing the buffer size allows the online policies to avoid excessive use of the costly server when there are sufficiently large number of arrivals from the free server. In the following, we present another class of online policies.

\subsection{Online Risky Policy}
In this part, we study a class of online policies where the costly server is used only if the buffer size is below a certain threshold. We call such policies ``risky'' as the risk of interruption is spread out across the whole trajectory unlike the ``safe'' policies.  Further, we constrain risky policies to possess the property that the action at a particular time should only depend on the buffer size at that time, i.e., such policies are \emph{stationary Markov} with respect to buffer size as the state of the system. The risky policy is formally defined below.

\begin{definition}\label{risky_def}
The online risky policy $\pi^T$ parameterized by the threshold value $T$ is given by
\begin{equation}\label{risky_policy}
    \pi^T(h_t) = \pi^T(x_t)= \left\{
                   \begin{array}{ll}
                     1, & \hbox{if $0< x_t  <  T$} \\
                     0, & \hbox{otherwise.}
                   \end{array}
                 \right.
\end{equation}
\end{definition}

\begin{lemma}\label{stop_lemma}
Let $x_t$ be the buffer size of a single-server system with arrival rate $R > 1$. Let the initial buffer size be $D$ and for any $T\geq D > 0$ define the following stopping times
\begin{equation}\label{stop_def}
\tau_T = \inf \{t > 0: x_t \geq T\},\quad  \tau_e = \inf \{t \geq 0: x_t \leq 0\}.
\end{equation}
Then
\begin{equation}\label{stop_prob}
\pr(\tau_e > \tau_T) = \frac{1-e^{-\bar r(R) D}}{1-\E[e^{-\bar r(R) x_{\tau_T}}|\tau_e > \tau_T]},
\end{equation}
where $\bar r(R)$ is defined in (\ref{r_bar}).
\end{lemma}

\begin{proof}
Let $Y(t) = e^{-\bar r(R) x_t}$. We may verify that $Y(t)$ is a martingale and uniformly integrable. Also, define the stopping time $\tau = \min\{\tau_T, \tau_e\}$. Since $R>1$, we have
$\pr(\tau \geq t) \leq \pr(0 < x_t < T) \rightarrow 0$, as $t\rightarrow \infty$. Hence, $\tau < \infty$ almost surely. Therefore, we can employ Doob's optional stopping theorem \cite{SP_book} to write
\begin{eqnarray*}
e^{-\bar r(R) D} &=& \E[Y(0)] = \E[Y(\tau)] \\
&=& \pr(\tau_e \leq \tau_T) \cdot 1 \\
&&+ \pr(\tau_e > \tau_T) \E[e^{-\bar r(R) x_{\tau_T}} |\tau_e > \tau_T].
\end{eqnarray*}
The claim immediately follows from the above relation after rearranging the terms.

\end{proof}

\begin{theorem}\label{risky_thm}
Let $\pi^T$ be the risky policy defined in Definition \ref{risky_def}.
For any $(D,\e) \in \mathcal R$, the policy $\pi^T$ is feasible if the threshold $T$ satisfies
\begin{equation}\label{threshold}
    T \geq T^* = \left\{
                   \begin{array}{ll}
                     \frac{1}{\a _1-\a _0}\big[\log\big(\frac{\beta}{\e}\big) - \a_0 D \big], & \hbox{if $D \geq \bar D$,} \\
                     \frac{1}{\a_1}\log\Big(\frac{\e + \beta(1-e^{-\a_1 D}) - 1}{\e - e^{-\a_1 D}}\Big), & \hbox{if $D \leq \bar D$,}
                   \end{array}
                 \right.
\end{equation}
where $\beta = \frac{\a_1}{\a_0 (1-\frac{\a_0}{2})}$ and $\bar D = \frac{1}{\a_1}\log\big(\frac{\beta}{\e}\big)$.

\end{theorem}

\begin{proof}
Let us first characterize the interruption probability of the policy $\pi^T$ when the initial buffer size is $D = T$. In this case,  by definition of $\pi^T$ the behavior of $x_t$ is initially the same as a single-server system with rate $R_1$ until the threshold $T$ is crossed. Hence, by Lemma \ref{stop_lemma} we have
\begin{eqnarray}
&& p^{\pi^{T}}(T) = \pr\Big(\min_{t \geq 0} x_t \leq 0 \big| x_0 = T\Big)  \nonumber \\
&& = \pr(\tau_e < \tau_{T})\cdot 1 \nonumber \\
&&+ \pr(\tau_{T} < \tau_e)\pr\Big(\min_{t \geq \tau_{T}} x_t \leq 0 \big| \tau_{T} < \tau_e, x_0 = T\Big) \nonumber\\
&&= \frac{e^{-\a_1 T} - \E[e^{-\a_1 x_{\tau_{T}}}|\tau_e > \tau_{T}]}{1-\E[e^{-\a_1 x_{\tau_{T}}}|\tau_e > \tau_{T}]} \nonumber \\
&& + \frac{\big(1-e^{-\a_1 T}\big)\pr\Big(\min_{t \geq \tau_{T}} x_t \leq 0 \big| \tau_{T} < \tau_e, x_0 = T\Big) }{1-\E[e^{-\a_1 x_{\tau_{T}}}|\tau_e > \tau_{T}]}. \nonumber\\
\label{eqn:step1}
\end{eqnarray}

Further,
\begin{eqnarray}\label{pint_cond}
 && \!\!\!\!\!\!\! \pr\Big(\min_{t \geq \tau_{T}} x_t \leq 0 \big| \tau_{T} < \tau_e, x_0 = T\Big) \nonumber \\
&& = \int_{T}^{T + 1} \pr\Big(\min_{t \geq \tau_{T}} x_t \leq 0 \big| x_{\tau_{T}}\Big) d\mu (x_{\tau_{T}}) \nonumber \\
&& \stackrel{(a)}{=} \int_{T}^{T + 1} \pr\Big(\min_{t \geq 0} x_t \leq 0 \big| x_0\Big) d\mu (x_0) \nonumber \\
&& \stackrel{(b)}{=} \int_{T}^{T + 1} \pr\Big(\min_{t \geq 0} x_t \leq 0 \big| \min_{t \geq 0} x_t \leq T, x_0\Big)\nonumber \\
&& \qquad \times \pr\big(\min_{t \geq 0} x_t \leq T| x_0\big) d\mu (x_0) \nonumber \\
&& \stackrel{(c)}{=} \int_{T}^{T + 1} p^{\pi^{T}}(T) e^{-\a_0(x_0 - T)}  d\mu (x_0) \nonumber \\
&& = \E[e^{-\a_0(x_{\tau_{T}} - T)}| \tau_{T} < \tau_e] p^{\pi^{T}}(T),
\end{eqnarray}
where $\mu$ denotes the conditional distribution of $x_{\tau_{T}}$ given $\tau_{T} < \tau_e$. Note that $x_{\tau_{T}} \in [T, T+1]$ because the size of the overshoot is bounded by one. Further, (a) follows from stationarity of the arrival processes and the control policy, (b) holds because a necessary condition for the interruption event is to cross the threshold $T$ when starting from a point $x_0 \geq T$. Finally (c) follows from (\ref{pd_exact}) and the definition of the risky policy. The relations (\ref{eqn:step1}) and (\ref{pint_cond}) together result in
\begin{equation}\label{pint_risky}
    p^{\pi^{T}}(T) = \frac{   e^{-\a_1 T} \big(1-\E_\mu[e^{-\a_1 (x_{\tau_{T}} - T)}]\big) }{ 1- \E_\mu[e^{-\a_0(x_{\tau_{T}} - T)}] + \kappa },
\end{equation}
where $\kappa = \E_\mu[e^{-\a_0 x_{\tau_{T}} -(\a_1 - \a_0)T}]  - \E_\mu[e^{-\a_1 x_{\tau_{T}} }] \geq 0$. Therefore, using the fact that
\begin{equation}\label{fact}
 1- x \leq e^{-x} \leq 1 - x +\frac{x^2}{2}, \quad \foral x \geq 0,
\end{equation}
we can provide the following bound
\begin{eqnarray} \label{pint_riksy_bd}
   p^{\pi^{T}}(T) &\leq & \frac{  e^{-\a_1 T} \big(\a_1 \E_\mu[x_{\tau_{T}} - T]\big)}{\a_0 \E_\mu[x_{\tau_{T}} - T]\Big(1 - \frac{\a_0}{2} \cdot \frac{\E_\mu[(x_{\tau_{T}} - T)^2]}{\E_\mu[x_{\tau_{T}} - T]} \Big)} \nonumber \\
&\leq& \frac{\a_1}{\a_0(1-\frac{\a_0}{2})} e^{-\a_1 T} = \beta e^{-\a_1 T},
\end{eqnarray}
where the last inequality holds since $0 \leq x_{\tau_{\bar D}} - \bar D \leq 1$.

Now we prove feasibility of the risky policy $\pi^{T^*}$ when $D > \bar D$. Observe that by (\ref{threshold}), $D > T^*$, hence the behavior of the buffer size $x_t$ is the same as the one in a single-server system with rate $R_0$ until the threshold $T^*$ is crossed. Thus
\begin{eqnarray*}
    p^{\pi^{T^*}}(D)  &=&   \pr\Big(\min_{t \geq 0} x_t \leq 0 \big| x_0 = D\Big) \\
&=& \pr\Big(\min_{t \geq 0} x_t \leq 0 \big| \min_{t \geq 0} x_t \leq T^*, x_0=D\Big)\\
&& \times \pr\big(\min_{t \geq 0} x_t \leq T^*| x_0=D\big) \\
&=&  p^{\pi^{T^*}}(T^*)  e^{-\a_0(D-T^*)} \\
&\leq&  \beta e^{-(\a_1-\a_0) T^* - \a_0 D} = \e,
\end{eqnarray*}
where the inequality follows from (\ref{pint_riksy_bd}), and the last equality holds by (\ref{threshold}).

Next we verify the feasibility of the policy  $\pi^{T^*}$ for  $D \leq \bar D$. In this case, $D \leq T^*$ and by definition of the risky policy the system behaves as a single-server system with arrival rate $R_1$ until the threshold $T^*$ is crossed or the buffer size hits zero (interruption). Hence, we can bound the interruption probability as follows
\begin{eqnarray*}
&&  p^{\pi^{T^*}}(D)  = \pr(\tau_e < \tau_{T^*}) \cdot 1 \\
&& + \pr(\tau_{T^*} < \tau_e) \pr\Big(\min_{t \geq \tau_{T^*}} x_t \leq 0 \big| \tau_{T^*} < \tau_e, x_0 = D\Big)\\
&& \stackrel{(a)}{=} 1 - \pr(\tau_{T^*} < \tau_e)\Big(1 - \E_\mu[e^{-\a_0(x_{\tau_{T^*}} - T^*)}] p^{\pi^{T^*}}(T^*) \Big)\\
&& \stackrel{(b)}{\leq} \frac{(\beta-1)(1 - e^{-\a_1 D})}{1-\E_\mu[e^{-\a_1 x_{\tau_{T^*}} }]} + 1 - \beta(1 - e^{-\a_1 D})\\
&& \stackrel{(c)}{\leq} \frac{(\beta-1)(1 - e^{-\a_1 D})}{1-e^{-\a_1 T^*} } + 1 - \beta(1 - e^{-\a_1 D}) \stackrel{(d)}{=} \e,
\end{eqnarray*}
where (a) follows from (\ref{pint_cond}), (b) can be verified after some manipulations by combining the result of Lemma \ref{stop_lemma} and (\ref{pint_risky}), and (c) holds since $\beta \geq 1$ and $x_{\tau_{T^*}} \geq T^*$. Finally, (d) immediately follows from plugging in the definition of $T^*$ from (\ref{threshold}).

Therefore, the risky policy $\pi^{T^*}$ is feasible by Definition \ref{feas_policy_def}. Observe that the buffer size under any policy $\pi^T$ of the form (\ref{risky_policy}) with $T\geq T^*$ stochastically dominates that of policy $\pi^{T^*}$, because $\pi^T$ switches to the costly server earlier, and stays in that state longer. Hence, $\pi^T$ is feasible for all $T\geq T^*$.
\end{proof}

Theorem \ref{risky_thm} facilitates the design of risky policies with a single-threshold structure, for any desired initial buffer size $D$ and interruption probability $\e$. For a fixed $\e$, when $D$ increases, $T^*$ (the design given by Theorem \ref{risky_thm}) decreases to zero. On the other hand, if $D$ decreases to $\frac{1}{\a_1}\log\big(\frac1\e\big)$ (the boundary of $\mathcal R$), the threshold $T^*$ quickly increases to infinity, i.e., the policy does not switch back to the free server unless a sufficiently large number of packets is buffered. Figure \ref{T*D_figure} plots $T^*$ and $D$ as a function of $D$ for a fixed $\e$. Observe that for large range of $D$, $T^* \leq D$, i.e., the costly server is not initially used. In this range, owing to the positive drift of $Q_t$, the probability of ever using the costly server exponentially decreases in $(D - T^*)$.

\begin{figure}[htbp]
\centering
\vspace{-0.15in}
  \includegraphics[width=3in]{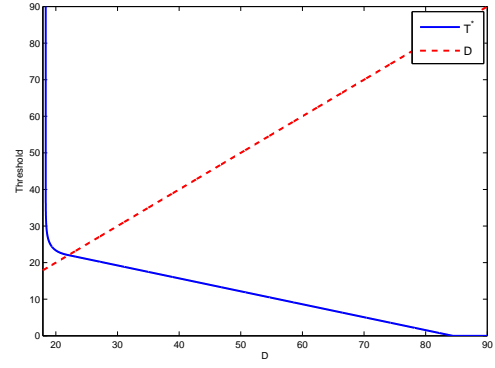}
  \caption{The switching threshold of the online risky policy as a function of the initial buffer size for $\e = 10^{-3}$ (See Theorem \ref{risky_thm}).}\label{T*D_figure}
\vspace{-0.1in}
\end{figure}

Next we compute relatively tight bounds on the expected cost of the online risky policy and compare with the previously proposed policies.

\begin{theorem}\label{risky_cost_thm}
For any $(D, \e) \in \mathcal R$, consider an online risky policy $\pi^{T^*}$ defined in Definition \ref{risky_def}, where the threshold $T^*$ is given by (\ref{threshold}) as function of $D$ and $\e$. If $D \geq \bar D$ then
\begin{equation}\label{risky_cost1}
  J^{\pi^{T^*}}(D,\e) \leq    \frac{\beta}{\a_1(R_1-1)}e^{-a_0(D- T^*)},
\end{equation}
and if $D \leq \bar D$
\begin{equation}\label{risky_cost2}
  J^{\pi^{T^*}}(D,\e) \leq \frac{1 - e^{-\a_1 D}}{(R_1-1)(1 - e^{-\a_1 T^*})}\Big(T^*+1+\frac{\beta}{\a_1}\Big) - \frac{D}{R_1-1},
\end{equation}
where $\beta = \frac{\a_1}{\a_0 (1-\frac{\a_0}{2})}$ and $\bar D = \frac{1}{\a_1}\log\big(\frac{\beta}{\e}\big)$.
\end{theorem}

\begin{proof}
Similarly to the proof of Theorem \ref{risky_thm}, we first consider the risky policy $\pi^T$ with the initial buffer size $T$. By definition of $\pi^{T}$, the costly server is used until the threshold $T$ is crossed. Thus the expected cost of this policy is bounded by the expected time until crossing the threshold plus the expected cost given that the threshold is crossed, i.e.,
\begin{equation*}
   J^{\pi^{T}}(T,\e) \leq \frac{\E[x_{\tau_{T}}] - T}{R_1-1} + E[e^{-\a_0(x_{\tau_{T}} - T)}]    J^{\pi^{T}}(T,\e),
\end{equation*}
where $\tau_{T}$ is defined in (\ref{stop_def}). The above relation implies
\begin{eqnarray}\label{risky_cost_bd}
     J^{\pi^{T}}(T,\e) &\leq&  \frac{1}{R_1-1} \cdot \frac{\E[x_{\tau_{T}} - T] }{1- E[e^{-\a_0(x_{\tau_{T}} - T)}]} \nonumber \\
&\leq& \frac{1}{R_1-1} \cdot \frac{1}{\a_0\Big(1 - \frac{\a_0}{2} \cdot \frac{\E_\mu[(x_{\tau_{T}} - T)^2]}{\E_\mu[x_{\tau_{T}} - T]}\Big)} \nonumber \\
&\leq& \!\! \frac{1}{\a_0(R_1-1)(1-\frac{\a_0}{2})} = \frac{\beta}{\a_1(R_1-1)},
\end{eqnarray}
where the second inequality follows from the fact in (\ref{fact}). Now for any $D\geq \bar D$ we can write
\begin{eqnarray*}
     J^{\pi^{T^*}}(D,\e)   &=&  \pr\Big(\min_{t \geq 0} x_t \leq T^* \big | x_0 = D \Big)      J^{\pi^{T^*}}(T^*,\e) \\
&=& e^{-a_0(D- T^*)} J^{\pi^{T^*}}(T^*,\e)
\end{eqnarray*}
where the inequality holds by Theorem \ref{pd_thm}. Combining this with (\ref{risky_cost_bd}) gives the result in (\ref{risky_cost1}).

If $D\leq \bar D$, the risky policy uses the costly server until the threshold $T^*$ is crossed at $\tau_{T^*}$ or the interruption event ($\tau_e$), whichever happens first. Afterwards, no extra cost is incurred if an interruption has occurred. Otherwise, by (\ref{risky_cost_bd}) an extra cost of at most $\frac{\beta}{\a_1(R_1-1)} $ is incurred, i.e.,
\begin{equation*}
     J^{\pi^{T^*}}(D,\e) \leq \E\big[\min\{\tau_e, \tau_{T^*}\}\big] + \pr(\tau_{T^*} < \tau_e) \frac{\beta}{\a_1(R_1-1)}.
\end{equation*}

By Doob's optional stopping theorem applied to the martingale $Z_t = x_t - (R_1-1)t$, we obtain

$$D = \pr(\tau_{T^*} < \tau_e) \E[x_{\tau_{T^*}}| \tau_{T^*} < \tau_e] -(R_1-1)\E\big[\min\{\tau_e, \tau_{T^*}\}\big], $$
which implies
$$\E\big[\min\{\tau_e, \tau_{T^*}\}\big] \leq \frac{\pr(\tau_{T^*} < \tau_e) (T^*+1) - D}{R_1-1}  $$

By combining the preceding relations we conclude that
$$  J^{\pi^{T^*}}(D,\e) \leq \frac{\pr(\tau_{T^*} < \tau_e)}{R_1-1}\Big(T^*+1+\frac{\beta}{\a_1}\Big) - \frac{D}{R_1-1},$$
which immediately implies (\ref{risky_cost2}) by employing Lemma \ref{stop_lemma}.

\end{proof}

In the following, we compare the expected cost of the presented policies using numerical methods, and illustrate that the bounds  derived in Theorems  \ref{offline_range_thm}, \ref{safe_thm} and \ref{risky_cost_thm} on the expected cost function are close to the exact value.

\subsection{Performance Comparison}

\begin{figure}[htbp]
\centering
\vspace{-0.15in}
  \includegraphics[width=3.5in]{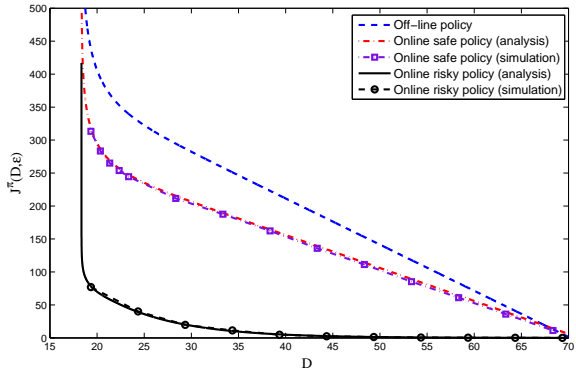}
  \caption{Expected cost (units of time) of the presented control policies as a function of the initial buffer size for interruption probability $\e = 10^{-3}$. The analytical bounds are given by Theorems \ref{offline_range_thm}, \ref{safe_thm} and \ref{risky_cost_thm}.}\label{cost_comparison_fig}
\vspace{-0.1in}
\end{figure}

Figure \ref{cost_comparison_fig} compares the expected cost functions of the off-line, online safe and online risky policies as a function of the initial buffer size $D$, when the interruption probability is fixed to $\e = 10^{-3}$, the arrival rate from the free server is $R_0 = 1.05$, and the arrival rate from the costly server is $R_c = R_1-R_0 = 0.15$. We plot the bounds on the expected cost given by Theorems \ref{offline_range_thm}, \ref{safe_thm} and \ref{risky_cost_thm} as well as the expected cost function numerically computed by the Monte-Carlo method.

Observe that the expected cost of the risky policy is significantly smaller that both online safe and off-line policies. For example, the risky policy allows us to decrease the initial buffer size from 70 to 20 with an average of $70 \times 0.15 \approx 10$ extra packets from the costly server.  The expected cost in terms of the number packets received from the costly server is 43 and 61 for the online safe and off-line policy, respectively.

Moreover, note that it is merely the existence of the costly server as a backup that allows us to improve the user's quality of experience without actually using too many packets from the costly server. For example, observe that the risky policy satisfies QoE metrics of $D = 35$ and $\e = 10^{-3}$, by only using on average about one extra packet from the costly server. However, without the costly server, in order to decrease the initial buffer size from 70 to 35, the interruption probability has to increase from $10^{-3}$ to about $0.03$ (see Theorem \ref{pd_thm}).

\section{Dynamic Programming Approach}
In this section, we present a characterization of the optimal association policy in terms of the Hamilton-Jacobi-Bellman (HJB) equation. Note that because of the probabilistic constraint over the space of sample paths of the buffer size, the optimal policy is not necessarily Markov with respect to the buffer size as the state of the system. We take a similar approach as in \cite{Chen04} where by expanding the state space, a Bellman equation is provided as the optimality condition of an MDP with probabilistic constraint. In particular, consider the pair $(x,p)$ as the state variable, where $x$ denotes the buffer size and $p$ represents the desired level of interruption probability. The evolution of $x$ is governed by the following stochastic differential equation
\begin{equation}\label{x_SDE}
    dx = -dt + dN^u,\quad x_0 = D,
\end{equation}
where $N^u$ is a Poisson counter with rate $R_u = R_0 + u\cdot R_c$. For any $(D,\e) \in \mathcal R$ and any optimal policy $\pi$, the constraint $p^{\pi}(D) \leq \e$ is active. Hence, we consider the sample paths of $p$ such that $p_0 = \e$ and $\E[p_t] = \e$ for all $t$, where the expectation is with respect to the Poisson jumps. Let $\p = p + dp$ if a Poisson jump occurs in an infinitesimal interval of length $dt$. Also, let $dp_0$ be the change in state $p$ is no jump occurs. Therefore,
$$0 = \E[dp] = R_u dt (\p - p) + (1-R_u dt)dp_0.$$
By solving the above equation for $dp_0$, we obtain the evolution of $p$ as
\begin{equation}\label{p_SDE}
    dp = (p - \p) (R_u dt - dN^u), \quad p_0 = \e.
\end{equation}

Similarly to the arguments of Theorem 2 of \cite{Chen04}, by principle of optimality we can write the following dynamic programming equation
\begin{equation}\label{HJB_eq1}
    V(x, p) = \min_{u \in\{0,1\}, \p \in[0,1]} \big\{ u dt +\E[V(x+dx, p+dp)] \big\}.
\end{equation}

If $V$ is continuously differentiable, by It\={o}'s Lemma for jump processes, we have
\begin{eqnarray*}
dV(x,p) &=& \frac{\partial V}{\partial x} (-dt) + \frac{\partial V}{\partial p} \cdot (p - \p)R_u dt \\
&& + \big(V(x+1,\p) - V(x,p)\big)dN^u,
\end{eqnarray*}
which implies the following HJB equation
\begin{eqnarray}\label{HJB}
   \frac{\partial V (x,p)}{\partial x} &=&  \min_{u \in\{0,1\}, \p \in[0,1]} \big\{ u+ \frac{\partial V}{\partial p} \cdot (p - \p)R_u \nonumber \\
&& + R_u \big(V(x+1,\p) - V(x,p)\big)   \big\}
\end{eqnarray}

The optimal policy $\pi$ is obtained by characterizing the optimal solution of the partial differential equation in (\ref{HJB}) together with the boundary condition $V(x,1) = 0$. Since such equations are in general difficult to solve analytically, we use the \emph{guess and check} approach, where we propose a candidate for the value function and verify that it nearly satisfies the HJB equation almost everywhere. Moreover, we show that the trajectories of $(x_t, p_t)$ steered by the optimal actions $(u^*, \p^*)$ lie in a one-dimensional invariant manifold, leading to the risky policy defined in Definition \ref{risky_def}.

For any $(x,p) \in \mathcal R$ define
\begin{equation}\label{T_xp}
    T(x,p) = \left\{
                   \begin{array}{ll}
                     \frac{1}{\a _1-\a _0}\big[\log\big(\frac{\theta}{p}\big) - \a_0 x \big], & \hbox{if $x \geq  \frac{1}{a_1}\log\big(\frac{\theta}{p}\big)$,} \\
                     \frac{1}{\a_1}\log\Big(\frac{p + \theta(1-e^{-\a_1 x}) - 1}{p - e^{-\a_1 x}}\Big), & \hbox{otherwise,}
                   \end{array}
                 \right.
\end{equation}
where $\theta = \frac{\a_1}{\a_0}$. The candidate solution for HJB equation (\ref{HJB}) is given by
\begin{equation}\label{candidate1}
    \bar V(x,p) = \frac{1}{\a_0(1-\frac{\a_0}{2})(R_1-1)}e^{-a_0(x- T(x,p))},
\end{equation}
when $x \geq  \frac{1}{a_1}\log\big(\frac{\theta}{p}\big)$, and
\begin{equation}\label{candidate2}
  \bar V(x,p) = \frac{p + \theta(1-e^{-\a_1 x}) - 1}{(R_1-1)(\theta -1)}\big(T(x,p)+\frac{\beta}{\a_1}\big) - \frac{x}{R_1-1},
\end{equation}
when $x < \frac{1}{a_1}\log\big(\frac{\theta}{p}\big)$. Note that the candidate solution is derived  from the structure if the expected cost of the risky policy (cf. Theorem \ref{risky_cost_thm}). We may verify that $\bar V$ satisfies the HJB equation (\ref{HJB}) for all $(x,p)$ such that $x \geq  \frac{1}{a_1}\log\big(\frac{\theta}{p}\big)$ or $x \geq  \frac{1}{a_1}\log\big(\frac{\theta}{p}\big)-1$, but for other $(x,p)$ the HJB equation is only approximately satisfied. This is due to bounding the overshoots, when computing the expected cost of the risky policy. The verification of HJB equation for our candidate solution is tedious but straightforward. We do not include it owing to space limitation.

\begin{theorem}
Let $\pi^*$ be the optimal association policy obtained from minimizing the right hand side of the HJB equation in (\ref{HJB}) for the value function given by (\ref{candidate1}) and (\ref{candidate2}). Then $\pi^*$ is a risky policy defined in Definition \ref{risky_def} with a threshold level $T(D,\e)$, where $D$ is the initial buffer size and $\e$ is the desired interruption probability.
\end{theorem}
\begin{proof}
We sketch the proof owing to space limitation. First, we can show that the optimal action $u^*(x,p)$ takes the following form
\begin{equation}\label{u_opt}
    u^*(x,p) = \left\{
                 \begin{array}{ll}
                   0, & \hbox{ if $x \geq \frac{1}{a_1}\log\big(\frac{\theta}{p}\big) $;} \\
                   1, & \hbox{otherwise.}
                 \end{array}
               \right.
\end{equation}

Moreover, we may verify that for the initial condition $(x_0, p_0) = (D,\e)$, the trajectory of $(x_t, p_t)$ steered by the optimal actions $(u^*, \p^*)$ is limited to a one-dimensional invariant manifold $\mathcal M(D,\e)$, where
\begin{eqnarray}
\mathcal M^{(D,\e)}= \Big\{(x,p): p = \theta e^{-\a_0x - (\a_1-\a_0)T}\cdot \mathbb I_{\{x \geq T(D,\e)\}} &&\nonumber \\
 + \frac{(\theta -1)e^{-\a_1 T(D,\e)} - e^{-\a_1 x } (1-\theta e^{-\a_1 T(D,\e)})}{1-e^{-\a_1 T(D,\e)}} \mathbb I_{\{x < T(D,\e)\}} \Big\},  &&\nonumber
\end{eqnarray}
where $T(D,\e)$ is given by (\ref{T_xp}). Therefore, by plugging the above relation back into (\ref{u_opt}), we can show that the optimal action $u^* = 0$ if and only if $x \geq T(D,\e)$, i.e., the optimal policy given by the HJB equation is of the form of the risky policy in Definition \ref{risky_def} with threshold $T = T(D, \e)$.
\end{proof}

\section{Conclusions and Future Work}\label{conclusion_sec}

In this paper we studied the problem of selecting the access-networks in a heterogeneous wireless environment for media streaming applications.  Our objective was to investigate the trade-offs between the network usage cost and the user's QoE requirements parameterized by initial waiting time and allowable probability of interruption in media playback.  We analytically characterized and compared the expected cost of both off-line and online policies, finally showing that a threshold-based onilne risky policy achieves the lowest cost. Moreover, we derived an HJB equation for the problem of finding the optimal deterministic policy formulated as an MDP with a probabilistic constraint, and verified that the the online risky policy nearly satisfies the HJB equation. Numerical analysis also confirmed our analytical results showing that merely the availability of a costly server used as a back-up significantly improves QoE of media streaming without incurring a significant usage cost.

In the future, we would like to study more accurate models of channel variations such as the two-state Markov model due to Gillbert and Elliot. In this work we focused on deterministic network association policies. Another extension of this work would consist of studying randomized control policies. Finally, we would like to study more of the peer-to-peer aspect of the system to understand the decision making at the system level.

\bibliographystyle{unsrt}
\bibliography{ISIT}

\begin{thebibliography}{10}

\bibitem{Lab09}
C.~Labovitz, D.~McPherson, and S.~Iekel-Johnson.
\newblock 2009 {I}nternet {O}bservatory report.
\newblock In {\em NANOG-47}, October 2009.

\bibitem{Lar10}
R.~Laroia.
\newblock {Future of Wireless? The Proximate Internet}.
\newblock In {\em Proc. of the Second International Conference on Communication
  Systems and Networks (COMSNETS)}, Bangalore, India, January 2010.

\bibitem{knocking}
Knocking.
\newblock http://knockinglive.com, 2010.

\bibitem{Chen04}
R.~Chen.
\newblock Constrained stochastic control with probabilistic criteria and search
  optimization.
\newblock In {\em Proc. 43rd IEEE Conference on Decision and Control}, December
  2004.

\bibitem{ZhoChi07}
Y.~P. Zhou, D.~M. Chiu, and J.~C.~S. Lui.
\newblock A simple model for analyzing {P2P} streaming protocols.
\newblock In {\em Proc. IEEE ICNP 2007}.

\bibitem{BonMas08}
T.~Bonald, L.~Massouli\'{e}, F.~Mathieu, D.~Perino, and A.~Twigg.
\newblock Epidemic live streaming: optimal performance trade-offs.
\newblock {\em SIGMETRICS Perform. Eval. Rev.}, 36(1):325--336, 2008.

\bibitem{YinSri10}
L.~Ying, R.~Srikant, and S.~Shakkottai.
\newblock {The Asymptotic Behavior of Minimum Buffer Size Requirements in Large
  P2P Streaming Networks}.
\newblock In {\em Proc. of the Information Theory and Applications Workshop (to
  appear)}, San Diego, CA, February 2010.

\bibitem{Acedanski05}
S.~Acedanski, S.~Deb, M.~M\'edard, and R.~Koetter.
\newblock How good is random linear coding based distributed networked storage.
\newblock In {\em {NetCod}}, 2005.

\bibitem{Rodriguez06}
C.~Gkantsidis, J.~Miller, and P.~Rodriguez.
\newblock Comprehensive view of a live network coding p2p system.
\newblock In {\em Proc. {ACM SIGCOMM}}, 2006.

\bibitem{wangLi07}
M.~Wang and B.~Li.
\newblock R$^2$: Random push with random network coding in live peer-to-peer
  streaming.
\newblock {\em IEEE JSAC, Special Issue on Advances in Peer-to-Peer Streaming
  Systems}, 25:1655--1666, 2007.

\bibitem{ChiZhang06}
H.~Chi and Q.~Zhang.
\newblock Deadline-aware network coding for video on demand service over p2p
  networks.
\newblock In {\em PacketVideo}, 2006.

\bibitem{KumAlt07}
D.~Kumar, E.~Altman, and J-M. Kelif.
\newblock {Globally Optimal User-Network Association in an 802.11 WLAN and 3G
  UMTS Hybrid Cell}.
\newblock In {\em Proc. of the 20th International Teletraffic Congress
  (ITC-20)}, Ottawa, Canada, June 2007.

\bibitem{CheLi08}
C.~Feng and B.~Li.
\newblock On large-scale peer-to-peer streaming systems with network coding.
\newblock In {\em Proceedings of the 16th ACM international conference on
  Multimedia}, Vancouver, Canada, October 2008.

\bibitem{ISIT_report}
A.~ParandehGheibi, M.~M\'edard, S.~Shakkottai, and A.~Ozdaglar.
\newblock Avoiding interruptions - {QoE} trade-offs in block-coded streaming
  media applications.
\newblock submitted to ISIT 2010, arXiv:1001.1937 [cs.MM].

\bibitem{SP_book}
I.~Karatzas and S.~Shreve.
\newblock {\em Brownian Motion and Stochastic Calculus}.
\newblock Springer, 1997.

\end{thebibliography}

%

\end{document}